\documentclass[11pt]{amsart} 
 \usepackage{hyperref}
\usepackage{amssymb, amsmath,amsthm}
 
 \usepackage{commath}
\usepackage{mathtools}
 \usepackage{pb-diagram}
 
\theoremstyle{plain}
\newcommand{\rank}{\mathop{{\rm rank}}\nolimits}

\newcommand{\mlabel}[1]{\marginpar{#1}\label{#1}}

\newcommand{\g}{{\mathfrak g}}

\newcommand{\h}{{\mathfrak h}}

\newcommand{\fa}{{\mathfrak a}}
\newcommand{\fb}{{\mathfrak b}}

\newcommand{\fg}{{\mathfrak g}}
\newcommand{\fh}{{\mathfrak h}}

\newcommand{\fj}{{\mathfrak j}}
\newcommand{\fk}{{\mathfrak k}}
\newcommand{\fl}{{\mathfrak l}}
\newcommand{\fm}{{\mathfrak m}}

\newcommand{\fq}{{\mathfrak q}}
\newcommand{\fp}{{\mathfrak p}}
\newcommand{\fr}{{\mathfrak r}}

\newcommand{\fz}{{\mathfrak z}}

\renewcommand{\:}{\colon}
\newcommand{\1}{\mathbf{1}}

\newcommand{\subeq}{\subseteq}

\newcommand{\into}{\hookrightarrow}

\newcommand{\R}{{\mathbb R}}
\newcommand{\C}{{\mathbb C}}

\newcommand{\bP}{{\mathbb P}}

\renewcommand{\H}{{\mathbb H}}

\newcommand{\bH}{{\mathbb H}}
\newcommand{\bS}{{\mathbb S}}

\renewcommand{\tilde}{\widetilde}

\newcommand{\SL}{\mathop{{\rm SL}}\nolimits}

\newcommand{\SO}{\mathop{{\rm SO}}\nolimits}
\newcommand{\SU}{\mathop{{\rm SU}}\nolimits}
\newcommand{\OO}{\mathop{\rm O{}}\nolimits}

\newcommand{\fsl} {\mathop{{\mathfrak{sl} }}\nolimits}

\newcommand{\su}  {\mathop{{\mathfrak{su} }}\nolimits}
\newcommand{\so}  {\mathop{{\mathfrak{so} }}\nolimits}

\newcommand{\ad}{\mathop{{\rm ad}}\nolimits}
\newcommand{\Ad}{\mathop{{\rm Ad}}\nolimits}

\newcommand{\Iso}{\mathop{{\rm Iso}}\nolimits}

\newcommand{\tr}{\mathop{{\rm tr}}\nolimits}

\newcommand{\Herm}{\mathop{{\rm Herm}}\nolimits}

\newcommand{\Aut}{\mathop{{\rm Aut}}\nolimits}

\newcommand{\Diff}{\mathop{{\rm Diff}}\nolimits}

\newcommand{\diag}{\mathop{{\rm diag}}\nolimits}

\newcommand{\id}{\mathop{{\rm id}}\nolimits}

\newcommand{\rad}{\mathop{{\rm rad}}\nolimits}
\renewcommand{\dim}{\mathop{{\rm dim}}\nolimits}
\newcommand{\codim}{\mathop{{\rm codim}}\nolimits}

\newcommand{\dS}{\mathop{{\rm dS}}\nolimits}

\renewcommand{\phi}{\varphi}

\newcommand{\nin}{\noindent} 
\newcommand{\oline}{\overline}

\newcommand{\res}{\vert}

\newcommand{\Spec}{{\rm Spec}}

\newcommand{\ssssarr}{\hbox to 15pt{\rightarrowfill}}
\newcommand{\sssarr}{\hbox to 20pt{\rightarrowfill}}
\newcommand{\ssarr}{\hbox to 30pt{\rightarrowfill}}
\newcommand{\sarr}{\hbox to 40pt{\rightarrowfill}}
\newcommand{\arr}{\hbox to 60pt{\rightarrowfill}}
\newcommand{\larr}{\hbox to 60pt{\leftarrowfill}}
\newcommand{\Arr}{\hbox to 80pt{\rightarrowfill}}

\def\theoremname{Theorem}
\def\propositionname{Proposition}
\def\corollaryname{Corollary}
\def\lemmaname{Lemma}
\def\remarkname{Remark}
\def\conjecturename{Conjecture} 

\def\definitionname{Definition}
\def\exercisename{Exercise}
\def\examplename{Example}
\def\examplesname{Examples}
\def\problemname{Problem}
\def\problemsname{Problems}

\def\@thmcounter#1{\noexpand\arabic{#1}}
\def\@thmcountersep{}
\def\@begintheorem#1#2{\it \trivlist \item[\hskip 
\labelsep{\bf #1\ #2.\quad}]}
\def\@opargbegintheorem#1#2#3{\it \trivlist
      \item[\hskip \labelsep{\bf #1\ #2.\quad{\rm #3}}]}
\makeatother
\newtheorem{theor}{\theoremname}[section]
\newtheorem{propo}[theor]{\propositionname}
\newtheorem{coro}[theor]{\corollaryname}
\newtheorem{lemm}[theor]{\lemmaname}

\newenvironment{thm}{\begin{theor}\it}{\end{theor}}
\newenvironment{theorem}{\begin{theor}\it}{\end{theor}}

\newenvironment{corollary}{\begin{coro}\it}{\end{coro}}

\newenvironment{lem}{\begin{lemm}\it}{\end{lemm}}
\newenvironment{lemma}{\begin{lemm}\it}{\end{lemm}}

\newtheorem{rema}[theor]{\remarkname}

\newenvironment{remark}{\begin{rema}\rm}{\end{rema}}
\newenvironment{rem}{\begin{rema}\rm}{\end{rema}}

\newtheorem{stepnow}[theor]{}

\newtheorem{defin}[theor]{\definitionname} 

\newenvironment{definition}{\begin{defin}\rm}{\end{defin}}
\newenvironment{defn}{\begin{defin}\rm}{\end{defin}}

\newtheorem{exerc}{\exercisename}[section]

\newtheorem{exa}[theor]{\examplename}

\newenvironment{example}{\begin{exa}\rm}{\end{exa}}
\newenvironment{ex}{\begin{exa}\rm}{\end{exa}}

\newtheorem{exas}[theor]{\examplesname}

\newtheorem{conj}[theor]{\conjecturename}

\newtheorem{pro}[theor]{\problemname}

\newtheorem{prs}[theor]{\problemsname}

\newcommand{\pmat}[1]{\begin{pmatrix} #1 \end{pmatrix}}

\renewcommand{\mlabel}{\label}

\newcommand{\rO}{\mathrm{O}}

\newcommand{\rI}{\mathrm{I}}
\newcommand{\rS}{\mathrm{S}}

\numberwithin{equation}{section}
\renewcommand{\phi}{\varphi}

\newcommand{\bproof}{\begin{proof}}
\newcommand{\eproof}{\end{proof}}

\newcommand{\lsl}{\mathfrak{sl}}

\usepackage{amsmath,amsfonts,amssymb}
\usepackage{amsbsy,amsgen,amsopn,amscd,amstext,amsxtra}
\usepackage{amsthm}

\begin{document}


\title{Symmetric spaces with dissecting involutions}

\author{Karl-Hermann Neeb}
\address{Department of Mathematics, Friedrich-Alexander-University of Erlangen-N\"urnberg, Cauerstrasse 11, 91058 Erlangen, Germany}
\email{neeb@math.fau.de}

\author{Gestur \'{O}lafsson}
\address{Department of Mathematics, Louisiana State University, Baton Rouge, LA 70803, U.S.A.}
\email{olafsson@math.lsu.edu}
\thanks{The research of K.-H. Neeb was partially
supported by DFG-grant NE 413/9-1. The research of G. \'Olafsson was partially supported by Simons grant 586106.}
\begin{abstract}
\noindent
An involutive diffeomorphism $\sigma$ of a connected smooth manifold $M$ 
is called dissecting if the complement of its fixed point 
set is not connected. 
Dissecting involutions on a complete Riemannian manifold 
are closely related to constructive quantum field theory through the
work of Dimock and Jaffe/Ritter on the construction of reflection positive 
Hilbert spaces. In this article we classify all 
pairs $(M,\sigma)$, where $M$ is an irreducible connected symmetric space, 
not necessarily Riemannian,
and $\sigma$ is a dissecting involutive automorphism. 
In particular, we show that the only irreducible, connected and 
simply connected Riemannian symmetric spaces with 
dissecting isometric involutions are $\bS^n$ and $\H^n$, 
where the corresponding fixed point spaces are 
$\bS^{n-1}$ and $\H^{n-1}$, respectively. 
\end{abstract}


\maketitle



\section{Introduction}
\noindent
In this article we classify dissecting involutions on irreducible symmetric spaces.
An involutive diffeomorphism $\sigma$ of a connected smooth manifold $M$ 
is called {\it dissecting} if the complement of the fixed point 
manifold $M^\sigma$ is not connected. Then $M\setminus M^\sigma$ has 
two connected components $M_\pm$ exchanged by~$\sigma$ 
and $\codim M^\sigma = 1$. 

We are interested in the situation where $M$ is a symmetric space 
and $\sigma$ is an automorphism of~$M$. 
Any connected symmetric space is of the form 
$M \cong G/H$, where $G$ is a connected Lie group acting faithfully 
by automorphisms 
on $M$ and there exists an involutive automorphism $\tau$ of $G$ 
for which $H$ is an open subgroup of $G^\tau=\{g\in G \: \tau (g)=g\}$, 
the subgroup of $\tau$-fixed points. Then $\tau$ induces an involutive 
automorphism, also denoted $\tau$, on the Lie algebra $\g$ of $G$, 
so that we obtain the {\it symmetric Lie algebra} $(\g,\tau)$. 
Any dissecting involutive automorphism $\sigma_M$ 
of $M$, fixing the base point 
$eH \in M$, leads to another involutive automorphism $\sigma$ of $\g$, 
commuting with $\tau$, such that 
\begin{equation}
  \label{eq:disscond}
\dim(\g^{-\tau} \cap \g^{-\sigma}) = 1. 
\end{equation}
For a simply connected symmetric space $M = G/H$, 
this condition is equivalent to $\sigma_M$ being dissecting. 
Accordingly, we call a triple $(\g,\tau,\sigma)$,  
consisting of a Lie algebra and two commuting involutive automorphisms 
$\tau$ and $\sigma$, {\it dissecting} if \eqref{eq:disscond} is satisfied. 

In this note we classify all dissecting triples $(\g,\tau,\sigma)$ 
which are {\it irreducible} in the sense that $\g$ contains no 
non-trivial $\tau$-invariant ideals. 
On the global level, this classifies connected, simply connected 
symmetric spaces with dissecting involutions (Theorem~\ref{thm:2.4}). 
For any connected symmetric space $M$ with a dissecting involution 
$\sigma_M$, the canonical lift $\sigma_{\tilde M}$ 
of $\sigma_M$ to the simply connected covering $\tilde M$ 
is also dissecting, 
but the converse is in general not true (cf.\ Example~\ref{ex:pn}). 
The classification is easy to state: 

\begin{thm} Up to coverings 
connected irreducible symmetric space 
with a dissecting involutive automorphism is a connected component of a quadric 
\[ Q := \{ x \in \R^{p+q} \: \beta_{p,q}(x,x) = 1\},
\quad \mbox{ where } \quad 
\beta_{p,q}(x,y) = \sum_{j = 1}^p x_j y_j - \sum_{j = p+1}^{p+q}  x_j y_j. \]
Here $G = \SO_{p,q}(\R)_0$ and 
$H_0 \cong \SO_{p-1,q}(\R)_0$ or $\SO_{p,q-1}(\R)_0$. 
Up to conjugation, the 
dissecting involution is given in the first case 
by $\sigma (x_1,\ldots ,x_{p+q})=(-x_1,x_2,\ldots ,x_{p+q})$ and 
in the second case by $\sigma (x_1,\ldots ,x_{p+q})
=(x_1,\ldots , x_{n-1},-x_{p+q})$.
\end{thm}

Our motivation to classify symmetric spaces with dissecting involutions 
comes from constructions in Quantum Field Theory and our work on
{\it reflection positivity}, see \cite{NO18} and the references therein. 
Starting  from a pair $(M,\sigma)$ of a 
complete Riemannian manifold $M$ and a dissecting involutive 
isometry one produces a  reflection positive Hilbert space 
\cite{Di04, JR08, AFG86, NO18} 
which provides the basis for constructing relativistic field theories 
on a Lorentzian manifold. Typical examples arise in physics 
from isometric dissecting involutions on 
flat euclidean space $\R^n$, the (positively curved) sphere $\bS^n$ 
and the (negatively curved) hyperbolic space $\bH^n$. 
Our classification implies in particular that these three types 
cover all (irreducible) Riemannian symmetric spaces with dissecting 
isometries. The special case of $M=\bS^n$ and related questions of
harmonic analysis and representation theory  is explored in \cite{NO19}.

The article is organized as follows. In Section \ref{sec:manifolds} 
we introduce the basic concepts. In Section~\ref{sec:dissec} 
we discuss several examples of homogeneous spaces 
with dissecting involutions, and Section~\ref{sec:classif} 
is devoted to the proof of the classification theorem.

 \section{Dissecting involutions on manifolds and symmetric spaces} \label{sec:manifolds} 
 \noindent
In this section we introduce the two basic concepts discussed in
this note, dissections and symmetric spaces. 

\subsection{Dissections on manifolds}
\begin{defn} \mlabel{def:diss} 
(a) A diffeomorphism $\sigma$ of a connected smooth 
manifold $M$ is called {\it dissecting} if the complement of the 
fixed point set $M^\sigma$ is not connected. 

\nin (b) A smooth involution $\sigma \: M \to M$ 
is  called a {\it reflection}  
if there exists a fixed point $p \in M^\sigma$ such that the fixed point space of 
the linear involution $T_p(\sigma)$ on~$T_p(M)$ is a hyperplane. 
\end{defn} 

 \begin{thm}\label{thm:dissManif} Let $M$ be a connected smooth manifold. 
\begin{itemize}
\item[\rm(i)] If $g$ is a complete Riemannian metric on $M$ and 
$\sigma$ is a dissecting isometry, then $\sigma$ is an involution. 
\item[\rm(ii)] If $\sigma \in \Diff(M)$ is an involution, then there exists a 
$\sigma$-invariant 
complete Riemannian metric on $M$. 
\item[\rm(iii)] If $\sigma \in \Diff(M)$ is a dissecting involution, 
then $M \setminus M^\sigma$ has two connected components $M_\pm$ with 
$\sigma(M_\pm) = M_\mp$ and each component of $M^\sigma$ is of codimension~one. 
\item[\rm(iv)] If $\sigma \in \Diff(M)$ is a reflection 
and $M$ is simply connected, then $\sigma$ is dissecting and its fixed point set 
$M^\sigma$ is a connected orientable hypersurface.
\end{itemize}
\end{thm}

\begin{proof} (i) follows from \cite[Lemma~2.7]{AKLM06}. 

\nin (ii) From \cite{NO61} we obtain a complete Riemannian 
metric $g$ on $M$. Then $\tilde g := g + \sigma^*g$ is also complete 
by the Hopf--Rinow Theorem because its closed balls are contained in closed 
$g$-balls, hence compact. 

\nin (iii), (iv): By (ii), we can assume that $\sigma$ is isometric with respect to a 
complete Riemannian metric. Hence (iii) follows from \cite[Lemma~2.7]{AKLM06} 
and (iv) from \cite[Thm.~2.8]{AKLM06}. 
\end{proof}

\subsection{Symmetric spaces} 

\begin{definition} \mlabel{def:ss} (a) Let $M$ be a smooth manifold and 
$\mu \: M \times M \to M, (x,y) \mapsto x \cdot y =: s_x(y)$ 
be a smooth map with the following properties: 
each $s_x$ is an involution for which $x$ is an  isolated fixed point and 
\begin{equation}
  \label{eq:symspcond}
 s_x(y \cdot z) = s_x(y)\cdot s_x(z) \quad \mbox{ for all } \quad x,y \in M.
\end{equation}
Then we call $(M,\mu)$ a {\it symmetric space}. 

\nin (b) A morphism of symmetric spaces $M$ and $N$ is a smooth
map $\varphi : M\to N$ such that $\varphi (x\cdot y)= \varphi (x)\cdot 
\varphi (y)$ for $x, y \in M$.
\end{definition}

As shown in \cite{Lo69}, connected symmetric spaces are homogeneous spaces 
of Lie groups and they arise from the following construction, which 
goes back to \'E. Cartan (see \cite{Hel78} for a detailed discussion).
Let $G$ be a Lie group and $\tau : G\to G$ be an involutive automorphism.
If $H\subset G^\tau$ is an open subgroup, then 
the homogeneous space $M := G/H$ is a symmetric space with respect to 
the involutions $s_{gH}(xH) := g\tau(g^{-1}x)H$.
Then $\tau$ defines an involution of $\g$, also 
denoted by $\tau$. We have a decomposition of $\g$ into eigenspaces $\fg=\fh\oplus \fq$
where $\fh=\g^{\tau}$ and $\fq=\g^{-\tau}$. Here 
$\fh$ is the Lie algebra of $H$, and
$\fq$ is identified with the tangent space $T_{eH}(G/H)$. 
 
Let $(M, (s_x)_{x \in M})$ be a connected symmetric space 
and $\sigma_M$ a dissecting involutive automorphism of $M$. 
Then the subset $M^{\sigma_M}$ of fixed points is non-empty 
and we may pick $x \in M^{\sigma_M}$. Writing $M \cong G/H$ for a 
connected transitive Lie group $G$ of automorphisms 
and the stabilizer group $H := G_x$, 
the reflection $s_x$ induces an involutive automorphism 
$\tau$ on the Lie algebra~$\g$ of $G$, and $\sigma_M(x) = x$ implies that 
$\sigma_M$ commutes with $s_x$ by \eqref{eq:symspcond}, so that it induces an involution 
$\sigma$ on $\g$ commuting with $\tau$. That $\sigma_M$ is dissecting 
implies that $(\g^{-\tau})^{\sigma}$ is a hyperplane in 
$\g^{-\tau} \cong T_x(M)$. This proves the ``only if'' direction in the following theorem. 
 
\begin{theorem} \mlabel{thm:2.4} Let $M=G/H$ be a simply connected 
symmetric space and $\sigma_M $ an involutive automorphism of $M$ leaving
the base point $x=eH$ invariant. Denote the corresponding involution
on $\fg$ by~$\sigma$. 
Then $\sigma_M$ is dissecting  if and only if $\dim (\g^{-\sigma}\cap \g^{-\tau})=1$.
\end{theorem}

\begin{proof} If $\sigma_M$ is dissecting, then we have just seen that 
$\dim (\g^{-\sigma}\cap \g^{-\tau})=1$. 
The other direction follows from Theorem \ref{thm:dissManif}(iv)  
because the involution $T_{eH}(\sigma_M)$ on $T_{eH}(M)$ 
corresponds to the involution $\sigma\res_{\fq}$ whose fixed point space 
$\fq^{\sigma} = \g^{-\tau} \cap \g^\sigma$ is a hyperplane. 
\end{proof}

\begin{defn} We call a triple $(\g, \tau,\sigma)$, consisting of a 
real Lie algebra $\g$ and two commuting involutions 
$\tau, \sigma \in \Aut(\g)$ {\it dissecting} if 
\begin{equation}
  \label{eq:discont}
\dim (\g^{-\tau} \cap \g^{-\sigma}) = 1.
\end{equation}
A triple $(\g,\tau,\sigma)$ is called 
{\it semisimple} if $\g$ is semisimple. 
We say that it is 
{\it irreducible} if $(\g,\tau)$ is irreducible, i.e., $\g^{-\tau} \not= \{0\}$,
 and $\{0\}$ and $\g$ are the only $\tau$-invariant ideals of~$\g$. 
\end{defn}

\begin{example} \mlabel{ex:pn}
We now show that a dissecting triple $(\g,\tau,\sigma)$ 
may also correspond to symmetric spaces $G/H$ on which the involution induced 
by $\sigma$ is not dissecting. By Theorem~\ref{thm:2.4} 
this can only happen if $G/H$ is not simply connected. 

For $n\ge 1$ and $e_1,\ldots , e_{n+1}$ the standard base for $\R^{n+1}$ we denote by  
$r_j$ the reflection $r_j (x)=x-2x_je_j$.  The 
$n$-dimensional sphere $M=\bS^n\cong \SO_{n+1}(\R)e_1 \cong 
\SO_{n+1}(\R) /\SO_n(\R)$ 
is a 
connected symmetric space in which the base point $e_1$ 
corresponds to the involution $\tau(x)=r_1 xr_1$.  At the same time, the involution $\sigma :=r_1$ is dissecting with fixed point space 
$M^\sigma=\{x \in \bS^n \: x_1=0 \} \cong \bS^{n-1}$, and 
$M_\pm=\{x\in \bS^n \: \pm x_1>0\}$ are the two connected components 
of $M \setminus M^\sigma$. 

Now replace $M$ by the projective space 
\[\bP^{n}=\bS^n/\{\pm \1\}\cong \SO_{n+1}(\R)/\rS(\rO_1(\R)\times \rO_n(\R)).\]
Using projective coordinates we
have
\[\sigma([\xi_1,\ldots , \xi_{n+1}])
=[-\xi_1,\xi_2,\ldots , \xi_{n+1}]=[\xi_1,-\xi_2,\ldots ,-\xi_{n+1}]\]
and  $(\bP^n)^\sigma \cong \bP^{n-1}$, the complement of the fixed point set 
is the open cell 
\[ M_+/\{\pm \1\} = M_-/\{\pm \1\}
= \{[\xi_1,\ldots ,\xi_{n+1}]\: \xi_1\not= 0\},\] 
hence connected. 
\end{example}

\begin{definition} Let $\fg$ be a Lie algebra and $\tau : \fg \to \fg$ 
an involutive automorphism. We write 
$\fh = \ker (\tau -\1)$ and $\fq= \ker (\tau +1)$, so that $\fg = \fh+\fq$. 

\nin (a) The pair $(\fg,\tau)$ as well as the pair $(\fg,\fh)$, is called a 
{\it symmetric pair}.  

\nin (b) The \textit{Cartan dual} 
of $(\fg,\tau)$ is the symmetric pair $(\fg^c, \tau^c)$ 
given by $\fg^c:= \fh+i\fq$ and $\tau^c(x+iy)
=x-iy$ for $x\in \fh$ and $y\in\fq$.

\nin (c) If $(\fg,\tau,\sigma)$ a triple, 
where $\tau$ and $\sigma : \fg\to \fg$ are commuting 
involutive automorphisms, then the \textit{Cartan dual} of 
$(\fg,\tau,\sigma)$ is 
the triple $(\fg^c,\tau^c,\sigma^c)$ with 
$\tau^c(x+iy) =x-iy$ and 
$\sigma^c(x+iy) =x-iy$ for $x \in \h, y \in \fq$. 
\end{definition}

\begin{rem} \mlabel{rem:irred} (a) 
If $(\g,\tau,\sigma )$ is an irreducible dissecting triple and  $\tau = \sigma$, then \eqref{eq:discont} implies 
$\dim \g^{-\tau} = 1$, so that $\g^{-\tau}$ is a one-dimensional ideal of $\g$. 
Hence $\tau\not=\sigma$ if $(\g,\tau,\sigma)$ is semisimple, 
irreducible and dissecting. 

\nin (b) If $(\g,\tau)$ is irreducible and $\g$ is not semisimple, 
then $\dim \g = 1$ and $\tau = -\id_\g$. In fact, the radical 
$\fr := \rad(\g)$ is a non-zero $\tau$-invariant ideal, hence equal to 
$\g$. As its commutator algebra $[\fg,\fg]$ is proper, it is zero, 
so that $\g$ is abelian. As every one-dimensional subspace of 
$\g^{-\tau} \not=\{0\}$ is an ideal, $\g = \g^{-\tau}$ is one-dimensional. 

\nin (c) Clearly, $(\g,\tau,\sigma)$ is dissecting if and only if 
$(\g,\sigma, \tau)$ is dissecting. 
It is easy to see that $(\g,\tau,\sigma)$ 
is dissecting
if and only if the Cartan dual ($c$-dual) $(\g^c, \tau^c,\sigma^c)$  
 is 
dissecting because $(\g^c)^{-\tau^c} \cap (\g^c)^{-\sigma^c} 
= i (\g^{-\tau} \cap \g^{-\sigma})$. 

\nin (d) The dual symmetric Lie algebra $(\g^c, \tau^c)$ is irreducible 
if and only if $(\g,\tau)$ is. In fact, $\tau$-invariant ideals 
of $\g$ are of the form $\fj = \fj_\fh \oplus \fj_\fq$, 
where 
\[  [\fh, \fj_\fh] \subeq \fj_\fh, \quad 
[\fh, \fj_\fq] \subeq \fj_\fq, \quad 
[\fq, \fj_\fh] \subeq \fj_\fq, \quad 
[\fq, \fj_\fq] \subeq \fj_\fh,\] 
and these conditions are equivalent to $\fj^c := \fj_\fh \oplus i \cdot\fj_\fq$ 
being an ideal of $\g^c$. 
\end{rem}

\section{Examples of 
symmetric spaces with dissecting involutions} 
\mlabel{sec:dissec}
\noindent
Let $(V,\beta)$ be a finite dimensional real vector space, 
endowed with a non-degenerate symmetric bilinear form 
$\beta$. Then every anisotropic element $x \in V$ defines an 
involution 
\[ \sigma_x(y) := y - 2 \frac{\beta(x,y)}{\beta(x,x)} x \]
fixing $x^\bot = \{ y \in V\: \beta(x,y) = 0\}$ pointwise, and for which 
$\R x$ is the $(-1)$-eigenspace. 

\begin{ex} \mlabel{ex:flat} If we consider $(V,\beta)$ as a flat semi-Riemannian 
symmetric space with respect to the point reflections 
$s_z(y) = 2z - y$, then  each $\sigma_x$ defines a dissecting 
isometric involution on $(V,\beta)$. 
\end{ex}

\begin{ex}  \mlabel{ex:quad} For $c \in \R^\times$, let 
\[ Q_c := Q_c(V,\beta) := \{ v \in V \: \beta(v,v) = c\} \] 
denote the corresponding quadric in $(V,\beta)$. 
Then $(Q_c, (-\sigma_x)_{x \in Q_c})$ 
is a symmetric space (Definition~\ref{def:ss}) 
and $\dim Q_c = \dim V -1$. 

If $\dim V = 1$, then $Q_c$ is discrete and contains at most two points, so
let us assume that $\dim V>1$. 
Let $M \subeq  Q_c$ be a connected component and let $x \in M$. 
Then $x^\bot$ contains a anisotropic element $y$ and 
the relations $\sigma_y(x) = x$ and $\sigma_y(Q_c) = Q_c$ show that 
$\sigma_y\res_M$ defines an involutive automorphism of the symmetric 
space $M$ for which $M^{\sigma_y} = M \cap y^\bot$ is non-empty. 
Moreover, $T_x(M \cap y^\bot) = T_x(M) \cap y^\bot 
= \{x,y\}^\bot$ is a hyperplane in $T_x(M) = x^\bot$. 
Therefore $M^{\sigma_y}$ is a hypersurface whose complement contains the two 
open subsets 
\[ M_\pm := \{ x \in M \: \pm \beta(x,y) > 0\} \]  
satisfying $\sigma_y(M_\pm) = M_\mp$. Thus $\sigma_y$ defines a dissecting 
involution of $M$ (Definition~\ref{def:diss}). 
That $M_\pm$ are in fact connected follows from a 
quick inspection of the possible quadrics in $3$-dimensional spaces. 

The space $(V,\beta)$ is isometric to some $(\R^{p+q},\beta_{p,q})$ 
with 
\[ \beta_{p,q}(x,y) = \sum_{j = 1}^p x_j y_j - \sum_{j = p+1}^{p+q}  x_j y_j.\] 
We can always assume that $c=1$ because dilation 
by $r \in \R^\times$ is an isomorphism of symmetric spaces 
from $Q_c$ to $Q_{r^2c}$ and 
the quadric $Q_{-1}$ in $(\R^{p+q},\beta_{p,q})$ is isomorphic to the 
quadric $Q_1$ in $(\R^{p+q},\beta_{q,p})$. 
In particular, all these quadrics have isomorphic symmetry groups  
$\OO_{p,q}(\R) \cong \OO_{q,p}(\R)$. 
Thus,
up to isomorphisms of symmetric spaces,
 we obtain the connected symmetric spaces 
$M^{p,q} \cong \SO_{p,q}(\R)_0/\SO_{p-1,q}(\R)_0$ 
as connected components of the quadric~$Q_1$. 
According to the sign of $\beta(y,y)$, we obtain two types of dissecting 
involutions on $M^{p,q}$. 
This leads to two types 
of symmetric spaces with dissecting involutions. 

Let $n := p + q -1 = \dim M \geq 1$. 
\begin{itemize}
\item For $q = 0$, we obtain the sphere $M^{n+1,0} = \bS^n\cong 
\SO_{n+1}(\R)/\SO_n(\R )$ with an isometric 
reflection $\sigma_{e_1}$ with fixed point set $\bS^{n-1}$ (an equator). 
\item For $p =1$, we obtain the hyperbolic space $M^{1,n} = \bH^{n}
\cong \SO_{1,n}(\R )/\SO_{n}(\R )$ 
with an isometric reflection with fixed point set $\bH^{n-1}$. 
\item For $q =1$, we obtain de Sitter space $M^{n,1}\cong 
 \dS^n\cong \SO_{n,1}(\R )/\SO_{n-1,1}(\R)$ and two possible 
fixed point sets, one isomorphic to $M^{n,0} = \bS^{n-1}\cong 
\SO_{n}(\R)/\SO_{n-1}(\R)$ 
and the other to  $M^{n-1,1} = \dS^{n-1}\cong \SO_{n-1,1}(\R)/\SO_{n-2,1}(\R)$. 
They correspond to the reflections $\sigma_y$ with 
$y = e_{n+1}$ and $e_n$, respectively. 
\end{itemize}
For $p, q > 1$ and the quadric $M^{p,q}$, 
we obtain fixed point manifolds of type $M^{p-1,q}$ and $M^{p,q-1}$. 
Note that $M^{p,q}$ is diffeomorphic to $\bS^{p-1} \times \R^q$, 
hence simply connected if $p \geq 3$. 

For the Lie algebra $\g = \so_{p,q}(\R)$ and an involution 
$\tau$ induced by a reflection $\sigma_{e_j}$ in $e_j^\bot$ for 
a basis vector $e_j$, we have 
\[ \g^c \cong 
\begin{cases}
\so_{p-1,q+1}(\R) & \text{ for } j \leq p \\ 
\so_{p+1,q-1}(\R) & \text{ for } j > p.  
\end{cases}\]
This follows easily by realizing $\g^c$ on the subspace 
$\R i e_j \oplus \sum_{k \not=j} \R e_k \subeq \C^{p+q}$.  
Therefore the class of dissecting triples obtained this way is 
stable under duality. 
\end{ex}

For $G = \SO_{p,q}(\R)_0$, we thus obtain dissecting triples 
$(\g,\tau,\sigma)$, where $\tau$ and $\sigma$ are induced by 
reflections in hyperplanes orthogonal to coordinate axes. 
All these examples are irreducible 
and we shall see that all irreducible semisimple dissecting 
triples are of this type. 

\begin{ex} \mlabel{ex:sl2} 
On the space $V := M_2(\R)$ of real $(2 \times 2)$-matrices, 
the determinant defines 
a quadratic form. From the relation 
$X^2 - (\tr X) X + \det(X)\1 = 0$, we derive that 
$\det(X) = \beta(X,X)$ for 
\[ \beta(X,Y) := \frac{1}{2}(\tr(X)\tr(Y) - \tr(XY)).\]  
This form is positive definite on $\R \1 \oplus \so_2(\R)$ and 
negative definite on its orthogonal complement 
$\fsl_2(\R) \cap \Herm_2(\R)$, hence of signature $(p,q) = (2,2)$. 
Therefore 
\[ M := \SL_2(\R) = \{ X \in M_2(\R) \: \det(X) = 1\} \] 
is a $3$-dimensional quadric and the 
$6$-dimensional group $G := \SO_{2,2}(\R)_0$ acts transitively on $M$. 
It contains the subgroup $\SL_2(\R)_L$ acting by left multiplications, 
and the subgroup $\SL_2(\R)_R$, acting by right multiplications. 
Clearly, $\SL_2(\R)_L \cap \SL_2(\R)_R = \{\pm \1\}$, so that 
$G = \SL_2(\R)_L \SL_2(\R)_R$ for dimensional reasons. 

Let $x := \1$, 
$y_1 := \pmat{0 & -1 \\ 1 & 0}$ and 
$y_2 := \pmat{1 & 0 \\ 0 & -1}$. Then 
$\beta(y_1, y_1) = 1$, 
$\beta(y_2, y_2) = -1$, and $y_{1/2}$ are both orthogonal to $x = \1$. 
We thus obtain two dissecting involutions $\sigma^M_{1/2}$ on $M = \SL_2(\R)$ from the 
corresponding orthogonal reflections. Their action on the Lie algebra 
$\fsl_2(\R) \cong T_\1(M)$ is given by 
\[ \sigma_1^\g(x) = x^\top \quad 
\mbox{ and } \quad
 \sigma_2^\g\pmat{a & b \\ c & -a} = \pmat{-a & b \\ c & a} 
= - I_{1,1} x I_{1,1} \quad \mbox{ for } \quad I_{1,1} = \pmat{1 & 0 \\ 0 & -1}.
\]The corresponding involutions on the symmetric space $M = \SL_2(\R)$ are 
\[ \sigma_1^G(g) = g^{\top} \quad \mbox{ and } \quad 
\sigma_2^G(g) = I_{1,1} g^{-1} I_{1,1}.\]
As 
$\sigma_1^G(g_1 g g_2^{-1}) = g_2^{-\top} \sigma_1^G(g) g_1^\top$ and 
$\sigma_2^G(g_1 g g_2^{-1}) = I_{1,1} g_2 I_{1,1}\sigma_2^G(g) I_{1,1} g_1^{-1} I_{1,1}$, 
the involutions induced on the Lie algebra 
$\g \cong \fsl_2(\R) \oplus \fsl_2(\R)$ are given by 
\[ \sigma_1(x,y) = (-y^\top, -x^{\top}) \quad \mbox{ and } \quad 
\sigma_2(x,y) = ( I_{1,1} y I_{1,1}, I_{1,1} x I_{1,1}). \] 
\end{ex}

\begin{ex} \mlabel{ex:su2} 
On the four-dimensional space $V := \bH = \R \1 \oplus \su_2(\C) 
\subeq M_2(\C)$ of quaternions, 
we consider the scalar product given by 
\[ \beta(x,y) := \frac{1}{2} \tr(xy^*).\] 
Then 
\[ M := \SU_2(\C) = \{ x \in \H \: \beta(x,x) = \det(x) = 1\} \] 
is a $3$-dimensional sphere. 
  
The $6$-dimensional group $G := \SO_4(\R)_0$ acts transitively on $M$. 
It contains the subgroup $\SU_2(\C)_L$ acting by left multiplications, 
and the subgroup $\SU_2(\C)_R$, acting by right multiplications. 
As above, $\SU_2(\C)_L \cap \SU_2(\C)_R = \{\pm \1\}$ and 
$G = \SU_2(\C)_L \SU_2(\C)_R$. 
As $x := \1$ and $y := \pmat{i  & 0 \\ 0 & -i}$ 
are orthogonal unit vectors, we 
obtain a dissecting involution $\sigma^M$ on $M = \SU_2(\C)$ from the 
orthogonal reflections $\sigma_y$. Its action on the Lie algebra 
$\su_2(\C) \cong T_\1(M)$ is given by 
\[  \sigma^\g\pmat{a i & b \\ -b^* & -a i} = \pmat{-a i & b \\ -b^* &  ai} 
=  - I_{1,1} x I_{1,1}.\] 
The corresponding involution on the symmetric space $M = \SU_2(\C)$ is 
$\sigma^G(g) = I_{1,1} g^{-1} I_{1,1},$ 
and the involution induced on the Lie algebra 
$\g \cong \su_2(\C) \oplus \su_2(\C)$ is  given by 
\[ \sigma(x,y) = ( I_{1,1} y I_{1,1}, I_{1,1} x I_{1,1}). \] 
\end{ex}

\section{The classification} 
\mlabel{sec:classif}
\noindent
In this subsection we show that all irreducible 
dissecting triples $(\g,\tau,\sigma)$ 
come from the quadrics as in Example~\ref{ex:quad}. 
Our strategy is to first reduce the
classification to the case where $\g$ is simple. Then we reduce it to
the case where $(\g,\tau)$ is Riemannian and eventually 
we have to classify certain triples of commuting involutions on 
$\so_n(\R)$. Here Riemannian means that there exists a Riemannian
symmetric space $M=G/H$ such that the Lie algebra of $G$ is $\fg$ and
the Lie algebra of $H$ is $\fh$.

We start with some simple observations concerning 
commuting involutions $\sigma$ and $\tau$ on a Lie algebra~$\g$. 
Let $\fh = \fg^\tau$, $\fq=\fg^{-\tau}$, $\fl = \fg^\sigma$, and $\fm=\fg^{-\sigma}$.
Then $\fg$ decomposes as
\[\fg = \fh\oplus \fq=\fl \oplus \fm = \fh_\fl \oplus \fh_\fm \oplus \fq_\fl \oplus \fq_\fm \]
where the indices $\fl$ and $\fm$, resp., indicate intersections 
with $\fl$ and $\fm$, respectively. 
We assume
that the symmetric Lie algebra $(\g,\tau)$ is irreducible 
and non-abelian (Remark~\ref{rem:irred}(b)). 

\begin{defn} \mlabel{def:4.1} If $\g$ is semisimple, then 
$\sigma$ and $\tau$ generate a finite subgroup of $\Aut(\g)$, 
so that \cite[Prop.~13.2.14]{HN12} implies 
that there exists a Cartan involution $\theta$ of $\g$, commuting with 
$\tau$ and $\sigma$. 
We write $\g^\theta = \fk$ and $\g^{-\theta} = \fp$ for the 
eigenspaces of $\theta$. 

The compact Lie algebra $\g^r := \fk \oplus i\fp$ is called 
the {\it compact dual of $\g$}. It carries three commuting involutions 
$\theta^r, \tau^r$ and $\sigma^r$, defined by extending each involution to a
complex linear involution on $\g_\C$ and then restrict to $\g^r$.

\end{defn} 

\begin{lem}
  \mlabel{lem:compdual} 
The triple $(\g^r, \tau^r, \sigma^r)$ is dissecting 
if and only if $(\g,\tau,\sigma)$ is. 
\end{lem}

\begin{proof}
First we observe that 
\[ (\g^r)^{-\tau^r} \cap (\g^r)^{-\sigma^r} 
=  (\fq \cap \fk \oplus i (\fq \cap \fp)) \cap 
(\fm \cap \fk \oplus i (\fm \cap \fp))
= (\fq_\fm \cap \fk) \oplus i(\fq_\fm \cap \fp).\] 
As $\fq_\fm$ is $\theta$-invariant, 
$\fq_\fm = (\fq_\fm \cap \fk) \oplus (\fq_\fm \cap \fp)$. 
Hence it is one-dimensional if and only if 
$(\g^r)^{-\tau^r} \cap (\g^r)^{-\sigma^r}$ is one-dimensional.
\end{proof}

We have the following possibilities 
for an irreducible non-abelian symmetric pair 
$(\g,\tau)$ (see \cite[p. 6]{F-J86} for a discussion
and \cite[pp.~9--11]{F-J86} for two commuting involutions): 
\begin{itemize}
\item[\rm (1)] $\fg$ is simple but not complex.\begin{footnote}{
We say that a real simple Lie algebra $\g$ is {\it complex} if 
there exists a complex structure $I$ on $\g$ commuting with $\ad \g$, 
so that $ix := Ix$ turns $\g$ into a complex Lie algebra. Recall 
that all complex simple Lie algebras are also simple as real Lie algebras.} 
  \end{footnote}
\item[\rm (2)] $\g$ is a simple complex Lie algebra and $\tau$ is complex linear.
\item[\rm (3)] $\g$ is a simple complex Lie algebra and $\tau$ is antilinear, 
i.e., a conjugation
with respect to the real form $\fg_1 = \g^\tau$, 
$\fg \cong \g_1 \oplus i\g_1$ and $\tau (x+iy)=x-iy$ for $x,y\in \fg_1$.
\item[\rm (4)] $\fg =\fg_1\oplus \fg_1$ with $\fg_1$ simple and 
$\tau(x,y)=(y,x)$. 
\end{itemize}
In fact, if $\g$ is not simple, then $\tau $ permutes the simple 
ideals of $\g$ in a non-trivial way, and since $\tau$ is an involution, 
irreducibility implies that $\g = \g_1 \oplus \g_2$ for two simple ideals 
satisfying $\tau(\g_1) = \g_2$. Then $\g_2 \cong \g_1$, and 
$(\g,\tau)$ takes the form (4). 

\begin{lem} \mlabel{lem:4.1} The pairs $(\fg,\tau)$ in {\rm(3)}, and the pairs 
$(\g = \g_1 \oplus \g_1, \tau)$ as in~{\rm (4)}, where 
$\g_1$ is simple but not complex, are $c$-dual to each other.
\end{lem}

\begin{proof} If $(\g,\tau)$ is of type (4), then 
$\g_\C \cong \g_{1,\C} \oplus \g_{1,\C}$. For 
$z = x + i y \in \g_{1,\C}$ we write $\oline z:= x -i y$. 
Then $\g^c=\{(z, \oline z) \: z \in \g_{1,\C}\}$, 
and $\g_{1,\C} \to \g^c, z \mapsto (z,\oline z)$ 
is an isomorphism of real Lie algebras. The Lie algebra $\g^c$ is simple 
if and only if $\g_1$ is not complex. 

If, conversely, $(\g,\tau) =(\g_{1,\C}, \tau)$, where 
$\tau$ is complex conjugation with respect to $\g_1$ as in~(3), 
then the inclusion $\g = \g_{1,\C} \into  \g_{1,\C} \oplus \g_{1,\C}, z 
\mapsto (z, \tau(z))$ extends to an isomorphism of complex 
Lie algebras $\g_\C \to \g_{1,\C} \oplus \g_{1,\C}$. 
Accordingly, 
\[ \g^c = \{ (x,x) + i(iy, -iy) 
= (x-y,x+y) \: x,y \in \g_1 \} \cong \g_1 \oplus \g_1
\quad \mbox{ with } \quad 
\tau^c(z,w) = (w,z).\] 
Since $\g \cong \g_{1,\C}$ is simple, $\g_1$ is not complex. 
\end{proof}

\begin{lem} \mlabel{lem:3.1} 
Let $\g_1$ be a simple real Lie algebra 
and consider on $\g := \g_1 \oplus \g_1$ the flip involution 
$\tau (x,y) = (y,x)$. Then, for any involution 
$\sigma$ of $\g$ commuting with $\tau$, there exists an involutive automorphism 
$\sigma_1$ of $\g_1$ such that either 
\[ \sigma (x,y) = (\sigma_1(x), \sigma_1(y)) \quad \mbox{ or } \quad 
   \sigma (x,y) = (\sigma_1(y), \sigma_1(x))\quad \mbox{ for } \quad x,y \in \g.\] 
\end{lem}

\begin{proof} As $\tau$ and $\sigma$ commute, 
$\sigma$ leaves the diagonal $\{(x,x)\: x\in \fg_1\}=\fg^\tau$ invariant. 
Hence there exists an involutive automorphism $\sigma_1 \in \Aut(\g_1)$ with 
\[ \sigma (x,x) = (\sigma_1(x), \sigma_1(x)) \quad \mbox{ for } \quad x \in \g_1.\] 
Now $\tilde \sigma (x,y)=(\sigma_1x, \sigma_1y)$ defines an involutive automorphism of 
$\g$ commuting with $\tau$. 
We show that either  $\sigma=\tilde \sigma $ or $\sigma=\tau\tilde\sigma$. 

We consider the automorphism 
$\gamma := \sigma\tilde\sigma \in \Aut(\g)$ which commutes with $\tau$ 
and fixes the diagonal $\g^{\tau}$ pointwise. 
Hence $\gamma$ preserves the anti-diagonal and there exists a linear map 
$c \: \g_1 \to \g_1$ with 
\[ \gamma(x,-x) = (c(x),-c(x))\quad \mbox{ for } \quad x \in \g_1.\] 
As $\gamma$ commutes with $\ad (x,x)$ for $x \in \g_1$, 
it follows that 
\[ c([x,y]) = [x,c(y)] = [c(x),y] \quad \mbox{ for } \quad x,y \in \g_1.\] 
We also have 
\begin{align*}
 ([c(x),c(y)], [c(x), c(y)]) 
&= [(c(x),-c(x)), (c(y), -c(y))]
= [\gamma(x,-x), \gamma(y,-y)] \\
&= \gamma([(x,-x),(y,-y)]) 
= \gamma([x,y], [x,y]) = ([x,y],[x,y]),
\end{align*}
so that 
\[ [x,y] = [c(x),c(y)] = c([x,c(y)]) = c^2([x,y]). \] 
As $\g$ is assumed to be simple, and hence $[\fg_1,\fg_1]=\fg_1$, it follows that 
$c^2 = \id_{\g_1}$. 

Since $c$ commutes with $\ad \g_1$, the eigenspaces 
$\ker(c \pm \1)$ are ideals of the simple Lie algebra~$\g_1$, so that 
$c \in \{\pm \id_{\g_1}\}$. We conclude that either 
$c = 1$, i.e., $\sigma = \tilde\sigma$, or 
$c = -1$, which means that $\sigma= \tau \tilde\sigma$. 
\end{proof}

 The triple $(\g,\tau, \sigma)$ is dissecting if and only if 
$\dim \fq_\fm=1$. Let $0 \not= x_0\in \fq_\fm$. 
Recall that $x\in \fg$ is said to be \textit{elliptic} 
if
$\ad x$ is semisimple with purely imaginary eigenvalues, 
and \textit{hyperbolic} if $\ad x$ is diagonalizable over $\R$. 

\begin{lemma}  \mlabel{lem:sigmatau1} 
For an irreducible non-abelian dissecting triple 
$(\fg,\tau ,\sigma)$, we have: 
\begin{itemize}
\item[\rm(i)] $\fg^{\sigma\tau}=\fh_\fl\oplus \R x_0$, $\fh_\fl$ 
is an ideal of $\fg^{\sigma\tau}$  and $x_0$ is central in $\g^{\sigma\tau}$. 
\item[\rm(ii)]   $x_0$ is either elliptic or hyperbolic.
\end{itemize}
\end{lemma}

Note that, if $x_0 \in \fq$ is elliptic, then $ix_0 \in \g^c 
= \fh \oplus i \fq$ is hyperbolic. 

\begin{proof}
(i) We have 
$\fg^{\sigma\tau}= \fh_\fl \oplus \fq_\fm=\fh_\fl \oplus \R x_0$. Hence
$[\fg^{\sigma\tau},\R x_0]= [\fh_\fl ,\fq_\fm] \subseteq \fq_\fm=\R x_0,$ 
so that $\R x_0$ is an ideal in $\fg^{\sigma\tau}$.
As $\fg^{\sigma \tau}$ is reductive,\begin{footnote}
{For any involution $\tau$ 
of a semisimple Lie algebra $\g$, the fixed point algebra 
is reductive in $\g$. Since there exists a Cartan involution 
$\theta$ commuting with $\tau$, this follows from 
\cite[Cor.~1.1.5.4]{Wa72}. }\end{footnote}
all its one-dimensional ideals are central. Therefore 
$x_0$ is central and thus $\fh_\fl$ also is an ideal. 

\nin (ii) According to \cite[Prop.~1.3.5.1]{Wa72}, 
 $x_0$ has a Jordan decomposition $x_0 =x_s+x_n$,  where
$x_s\in \g$ is semisimple and $x_n\in \g$ is nilpotent. 
Applying $\sigma$, the uniqueness of the Jordan decomposition 
and 
\[ x_s + x_n = x = - \sigma(x) = - \sigma(x_s) - \sigma(x_n)\] 
imply that $\sigma(x_s) = -x_s$ and
$\sigma (x_n)=-x_n$. We likewise obtain $\tau(x_s) = - x_s$
and $\tau (x_n)=-x_n$. As $\dim \fq_\fm = 1$,  we have
$x_0=x_s$ or $x_0=x_n$. 

As  $\fg^{\sigma\tau}$ is reductive in $\fg$, the Lie algebra 
decomposes into simple $\g^{\sigma\tau}$-submodules. If the central 
element $x_0\in \g^{\sigma\tau}$ is nilpotent, 
it acts trivially on all these simple submodules, so that 
$\ad x_0 = 0$. As $\g$ is semisimple and $x_0 \not=0$, 
this cannot happen, and we conclude 
that $x_0$ is semisimple. 

Now we decompose $x_0=x_e+x_h$ where $x_e$ and $x_h$ commute,
$x_e$ is elliptic and $x_h$ is hyperbolic. 
The uniqueness of the decomposition implies as above that 
$x_e, x_h \in \fq_\fm$, hence that $x_0$ is either elliptic or hyperbolic. 
\end{proof}

Let us now reduce the classification to the case where $\fg$ is simple. 
This is achieved by the following lemma, which reduces the case where
$\g$ is not simple to $\su_2(\C) \cong \so_3(\R)$ and $\fsl_2(\R) \cong \so_{1,2}(\R)$ 
which is discussed in Examples~\ref{ex:sl2} and~\ref{ex:su2}. 
We first take a closer look at these two Lie algebras. 

\begin{remark} \mlabel{rem:4.7}
The involutions on $\su_2(\C)$ and $\fsl_2(\R)$ 
are determined by their one-dimensional fixed point algebras 
$\g_1^{\sigma_1} = \R z_0$, where 
$z_0$ can be either elliptic of hyperbolic for $\fsl_2(\R)$,  
and $z_0$ is elliptic for $\su_2(\C)\cong \so_3(\R)$. 
As $\Aut(\su_2(\C)) = \Aut(\so_3(\R)) \cong \SO_3(\R)$, involutive 
automorphisms correspond to non-trivial involutions $\sigma \in \SO_3(\R)$. 
As $\det \sigma = 1$, these reflections  have 
$1$-dimensional fixed point spaces, hence are mutually conjugate, 
so that there is only one equivalence class for $\su_2(\C)$. 
We also observe that complex conjugation on $\su_2(\C)$ coincides 
with $x\mapsto -x^\top$ and with conjugation by 
$\begin{pmatrix} 0 & 1\\ -1 & 0\end{pmatrix}$.  
\end{remark}

\begin{lem} \mlabel{lem:4.4} Assume that $(\g,\tau,\sigma)$ is an irreducible 
non-abelian dissecting triple. If $\g$ is not simple, then 
$(\g,\tau,\sigma)$ is equivalent to 
\[ (\g_1 \oplus \g_1, \tau,\sigma) \quad \mbox{ with } \quad 
\tau (x,y) = (y,x),\quad \sigma (x,y) = (\sigma_1(y), \sigma_1(x)),\] 
where either 
\begin{itemize}
\item[\rm(a)] $\g_1 = \su_2(\C) \simeq \so_{3}(\R)$ and $\sigma_1(x) = \oline x$.  
\item[\rm(b)] $\g_1 = \fsl_2(\R)\simeq \so_{2,1}(\R)$ and 
$\sigma_1(x) = - x^\top$ or $\sigma_1(x) = \rI_{1,1} x \rI_{1,1}.$
\end{itemize}
\end{lem}

\begin{proof} Since $\fg$ is not simple, 
$(\fg,\tau)$ is equivalent to a pair of 
the form $\fg =\fg_1\oplus \fg_1$ with $\fg_1$ simple and 
$\tau (x,y)= (y,x)$.  Lemma~\ref{lem:3.1} 
shows that there exists an involutive automorphism 
$\sigma_1$ of $\g_1$ such that either 
\[ \sigma (x,y) = (\sigma _1(x), \sigma_1(y)) \quad \mbox{ or } \quad 
   \sigma (x,y) = (\sigma_1(y), \sigma_1(x))\quad \mbox{ for } \quad x,y \in \g.\] 

Let $\widetilde{\sigma}(x,y)=(\sigma_1(x),\sigma_1(y))$ be as in the
proof of Lemma \ref{lem:3.1}. We then have 
\[ \fq_\fm =
\begin{cases}
\{(x,-x)\: \sigma_1 x=-x\} \cong \g_1^{-\sigma_1} & \text{ for } \quad 
\sigma = \tilde\sigma \\  
\{(x,-x)\: \sigma_1 x=x\} \cong \g_1^{\sigma_1} & \text{ for } \quad 
\sigma = \tau \tilde\sigma.
\end{cases}\]
Write $x_0 = (y_0,-y_0)$. In the first case 
$\fg_1^{-\sigma_1}=\R y_0$ is 
one dimensional. Hence $\fg_1=\fg_1^{\sigma_1}\oplus \R y_0$ and 
\[ [\fg_1,y_0] = [\fg_1^{\sigma_1},y_0]\subseteq \R y_0 . \]
It follows that $\R y_0$ is an ideal, contradicting the assumption that $\fg_1 $ is simple. Hence we are in the second case where 
$\sigma = \tau \tilde\sigma$ 
and $\fq_\fm \cong \fg_1^{\sigma_1} = \R y_0$ is one-dimensional. 
It therefore remains to determine the symmetric pairs 
$(\g_1,\sigma_1)$, where $\g_1$ is simple and $\g_1^{\sigma_1}$ is one-dimensional. 

\nin{\bf Case 1:} $\fg_1^{-\sigma_1}$ is 
an irreducible module of $\R y_0 =\g_1^{\sigma_1}$. 
Then $\dim \fg_1^{-\sigma_1} \leq 2$, and since $\dim \g_1 \geq 3$, 
it follows that $\g_1$ is $3$-dimensional, 
hence isomorphic to $\su_2(\C)$ or $\fsl_2(\R)$. 
Up to equivalence, there are only two  equivalence classes of 
three-dimensional symmetric Lie algebras $(\g_1, \tau_1)$ 
for which $\ad y_0$ is irreducible on $\g_1^{-\sigma_1}$, 
one for $\g_1 \cong \su_2(\C)$ and one for 
$\g_1 \cong \fsl_2(\R)$ (Remark~\ref{rem:4.7}). Here $y_0$ is elliptic. 

\nin{\bf Case 2:} $\fg_1^{-\sigma_1}$ is not irreducible under 
$\ad y_0$. By \cite[Lem.~1.3.4, p.~13]{HO97}, 
$\g_1$ is $3$-graded, $\g_1 =V_{-1} \oplus V_0 \oplus V_1$, 
where $V_0 = \g_1^{\sigma_1} = \R y_0$ 
and $V_{\pm 1}$ are irreducible $\g^{\sigma_1}$-submodules of 
$\g_1^{-\sigma_1}$ which are abelian Lie algebras.
Then $D v_j := j v_j$ for $v_j \in V_j$ defines a derivation 
of $\g_1$. As $\g_1$ is semisimple, $D = \ad h$ for some $h \in \g_1$. 
Since $D$ commutes with $\sigma_1$, we have $\sigma_1(h) = h$. 
Therefore $h \in  \R y_0$,  and in particular $y_0$ is hyperbolic. 
Irreducibility of $V_{\pm 1}$ thus implies 
$\dim V_1=\dim V_{-1}=1$. 
Hence $\g_1$ is a $3$-dimensional non-compact simple Lie algebra, 
so that  $\fg_1\simeq \lsl_2(\R )$. 
\end{proof}

We next discuss the case where $\g$ is complex simple. 

\begin{corollary} \mlabel{cor:4.9} If the triple $(\g,\tau,\sigma )$ is dissecting 
and $\fg$ is complex simple, 
then $\g \simeq \lsl_2(\C)\simeq \so_{3,1}(\R)$ and
$\tau$ is the conjugation with respect to $\su_2(\C)\simeq \so_3(\R)$
or $\lsl_2(\R)\simeq \so_{2,1}(\R)$.  In this case  $\g^r \cong \su_2(\C) \oplus 
\su_2(\C) \cong \so_3(\R) \oplus \so_3(\R) \cong \so_4(\R)$. 
\end{corollary}

\begin{proof} Since $\g^{-\tau} \cap \g^{-\sigma}$ is one-dimensional real, 
either $\tau$ or $\sigma$ is not complex linear. Since 
$(\g,\sigma,\tau)$ is also dissecting, we may 
assume that $\tau$ is antilinear. 
Then $\g_1 := \g^\tau$ is a real form of $\g$ and $\g \cong \g_{1,\C}$. 
By Lemma~\ref{lem:4.1}, $\g^c := \g^\tau \oplus i \g^{-\tau} 
\cong \g_1 \oplus \g_1$ 
with $\tau^c(z,w) = (w,z)$, 
and $(\g^c, \tau^c, \sigma^c)$ is dissecting by Remark~\ref{rem:irred}(c). 
Further, $(\g^c,\tau^c)$ is irreducible (Remark~\ref{rem:irred}(d)), so that 
Lemma~\ref{lem:4.4} implies that $\g \cong \g_{1,\C} 
\cong \fsl_2(\C) \cong \so_{1,3}(\R)$, 
where $\tau$ corresponds to the involution on $\so_{1,3}(\R)$ 
with fixed point algebra $\so_3(\R)$ or $\so_{1,2}(\R)$, 
the two real forms of~$\g$. The compact dual of $\g = \fsl_2(\C)$ 
is $\g^r = \su_2(\C) \oplus \su_2(\C)
\cong \so_3(\R) \oplus \so_3(\R) \cong \so_4(\R)$ 
(Lemma~\ref{lem:4.1}). 
\end{proof}

\begin{rem} \mlabel{rem:grsimp} 
By Lemma~\ref{lem:4.1}, the 
compact dual  $\g^r$ (Definition~\ref{def:4.1}) is not simple if and only if 
$\g$ is a complex simple Lie algebra. 
\end{rem}

With Lemma~\ref{lem:4.4} and Corollary~\ref{cor:4.9}, 
we have largely reduced our problem to the 
case where $\g$ is simple and not complex, so that the 
compact dual $\g^r$ is simple (Remark~\ref{rem:grsimp}). 
The final step consists in an inspection of the Riemannian case. 
We recall from Example~\ref{ex:pn} the reflection 
$r_j(x) = x - 2 x_j e_j$ in the hyperplane $e_j^\bot \subeq \R^n$. 

\begin{thm} \mlabel{thm:6.6}
Let $(\g,\tau,\sigma)$ be an irreducible non-abelian dissecting triple such that 
$(\g,\tau)$ is Riemannian with $\dim (\g/\fh) = n$. 
Then $(\g,\tau,\sigma)$ is equivalent to one of the following types: 
\begin{itemize}
\item[\rm(C)] $\g \cong \so_{n+1}(\R)$ with $\tau(x) = r_1 x r_1$ and 
$\sigma(x) = r_{n+1} x r_{n+1}$. 
\item[\rm(NC)] $\g \cong \so_{1,n}(\R)$ with $\tau(x) = r_1 x r_1$ and 
$\sigma(x) = r_{n+1} x r_{n+1}$. 
\end{itemize}
\end{thm}

\begin{proof} That $(\g,\tau)$ is Riemannian implies that 
$\g^\tau$ is a compact Lie algebra. 
If $\g$ is not simple, then Lemma~\ref{lem:4.4} implies that 
$\g \cong \so_3(\R) \oplus \so_3(\R) \cong \so_4(\R)$ 
with $\tau(x,y) = (y,x)$ and 
$\sigma(x,y) = (\oline y, \oline x)$, which is also equivalent to 
$\sigma(x,y) = ( I_{1,1} y I_{1,1}, I_{1,1} x I_{1,1})$ 
(cf.~Remark~\ref{rem:4.7}). 
As we have seen in Example~\ref{ex:su2}, this implies that 
(C) holds with $n = 3$. 

We may therefore assume that $\g$ is simple. 
That $(\g,\tau)$ is 
Riemannian means that either $\g$ is compact or that $\tau$ is a 
Cartan involution. 
In the reductive subalgebra 
$\fh^a := \g^{\sigma\tau} = \fh_\fl \oplus \fq_\fm,$ 
the one-dimensional subspace $\fq_\fm$ 
is central by Lemma~\ref{lem:sigmatau1}. 
Let $\fb := \fz_\g(\fq_\fm) \cap \g^{-\sigma\tau}$ be the centralizer of~$\fq_\fm$ 
in $\fq^a := \g^{-\sigma\tau} = \fh_\fm \oplus \fq_\fl$. 
Then $\fb$ is invariant under $\fh^a$ 
and both involutions $\sigma$ and $\tau$. 
As the Cartan--Killing form $\kappa$ 
is non-degenerate on $\tau$-invariant subspaces, 
we have 
$\fq^a = \fb \oplus \fb_1$, where 
$\fb_1$ is the $\kappa$-orthogonal complement of $\fb$ in $\fq^a$. 
Then the subspace $[\fb,\fb_1] \subeq \fh^a$ satisfies 
\[ \kappa(\fh^a, [\fb, \fb_1]) 
=  \kappa([\fh^a, \fb], \fb_1) 
\subeq   \kappa(\fb, \fb_1) = \{0\},\] 
and since $\kappa$ is non-degenerate on the $\tau$-invariant subalgebra 
$\fh^a$, we obtain 
$[\fb,\fb_1] = \{0\}$. This implies that 
$\fb + [\fb,\fb]$ is a $\tau$-invariant ideal of $\g$ because 
it is invariant under $\fh^a$, $\fb$ and $\fb_1$. 
Hence the irreducibility of $(\g,\tau)$ and $\fq_\fm \not=\{0\}$ entail that 
$\fb = \{0\}$. This shows that 
\begin{equation}
  \label{eq:ha-cent}
 \fh^a = \fz_\g(\fq_\fm).
\end{equation}
Hence  the one-dimensional subspace $\fa := \fq_\fm$ is maximal 
abelian in $\fq$, so that  $\rank_\R(\g,\tau) = 1$. 

\nin {\bf Case 1.} $\g$ is not compact and $\tau$ is a Cartan involution: 
Let $0 \not= x_0 \in \fq_\fm$. Then $\ad x_0$ is diagonalizable and, for 
every eigenvalue $\lambda \in \R^\times$, the sum 
\[ \fb_\lambda := \g(\ad x_0, \lambda) + \g(\ad x_0, -\lambda) \] 
of the eigenspaces is a subspace of $\fq^a$,  invariant under 
$\fh^a$ and $\tau$. The same 
argument as above now implies that $\fb_\lambda + [\fb_\lambda, \fb_\lambda]$ 
is a $\tau$-invariant ideal of $\g$. As $\fb_\lambda \not=\{0\}$ by 
assumption, it follows that $\fq^a = \fb_\lambda$, so that 
$\Spec(\ad x_0) = \{0, \lambda, - \lambda\}$. We conclude that the 
restricted root system of $(\g,\tau)$ is $\Sigma = \{ \pm \lambda\}$ of 
type $A_1$. 
According to the classification of real non-compact simple 
Lie algebras (\cite[pp.~520--543]{Kn96}), 
this implies that $\g \cong \so_{1,n}(\R)$ for $n \geq 2$ and that $M \cong \bH^n$. 
\begin{footnote}
{To derive this from the tables in \cite{Hel78}, we 
note that $\su^*_{4}(\C) \cong \fsl_2(\bH) \cong \so_{1,5}(\R)$  
(cf.\ \cite[p.~615]{HN12}).}
\end{footnote}

To see how $\sigma$ looks like, we recall that it induces 
a dissecting isometric involution on the hyperbolic space 
$\bH^n \cong \SO_{1,n}(\R)/\SO_n(\R)$ which fixes the base point, 
say $e_1 \in \R^{1+n}$. As $\Iso(\bH^n)_{e_1} \cong \OO_n(\R)$ 
and all dissecting involutions in $\OO_n(\R)$ are conjugate to $r_{n+1}$, 
the assertion follows in this case. 

\nin {\bf Case 2.} $\g$ is compact simple: 
Then $\g^c = \fh \oplus i\fq$ is a non-compact simple Lie algebra 
with Cartan involution $\tau^c$, and $(\g^c, \tau^c, \sigma^c)$ is dissecting. 
Therefore Case $1$ implies that $\g^c \cong \so_{1,n}(\R)$, so that 
$\g \cong \so_{1 + n}(\R)$ with $\tau(x) = r_1 x r_1$ 
and $\sigma(x) = r_{n+1} x r_{n+1}$. 
\end{proof}

 \begin{theorem}{\rm (Classification of dissecting symmetric triples)}\label{thm:Classifi}
Every irreducible dissecting triple $(\g,\tau, \sigma)$ 
is equivalent to one of the type 
$(\so_{p,q}(\R), \tau, \sigma)$, where 
$\tau$ and $\sigma$ are induced by reflections corresponding to 
orthogonal anisotropic vectors in $\R^{p,q}$. 
 \end{theorem}

\begin{proof} That the specified triples 
$(\so_{p,q}(\R),\tau,  \sigma)$ are dissecting 
follows from Example~\ref{ex:quad}. 

If $(\g,\tau, \sigma)$ is an irreducible dissecting triple 
for which $\g$ is not simple, then 
$\g = \g_1 \oplus \g_1$ for 
$\g_1 \cong \fsl_2(\R)$ or $\su_2(\C)$ as in 
Lemma~\ref{lem:3.1}. 
For $\g_1 \cong \fsl_2(\R)$, we have seen in 
Example~\ref{ex:sl2} that these triples are of the form 
$(\so_{2,2}(\R), \tau, \sigma)$ as required. 
If $\g_1 = \su_2(\C)$, then 
$(\g,\tau)$ is Riemannian, so that Theorem~\ref{thm:6.6} 
implies that $(\g,\tau,\sigma)$ is equivalent to a triple 
$(\so_4(\R), \tau,\sigma)$ of the required form. 

We may therefore assume that $\g$ is simple. 
If $\g$ is a complex Lie algebra, then Corollary~\ref{cor:4.9} 
implies that $\g^r \cong \so_4(\R)$, and 
$(\g^r, \tau^r, \sigma^r)$ is dissecting by 
Lemma~\ref{lem:compdual}.  By Theorem~\ref{thm:6.6} it is 
equivalent to a triple with $\tau^r(x) = r_1 x r_1$ and $\sigma^r(x) = r_4 x r_4$.  

If $\g$ is simple and not complex, then 
$(\g^r, \tau^r, \sigma^r)$ is a compact dissecting triple, 
where $\g^r$ is simple (Remark~\ref{rem:grsimp}), hence isomorphic to 
$(\so_{n+1}(\R), \tau^r, \sigma^r)$, where 
$\tau^r(x) = r_1 x r_1$ and $\sigma^r(x) = r_{n+1} x r_{n+1}$ 
(Theorem~\ref{thm:6.6}). 
To see how $(\g,\tau,\sigma)$ looks like, 
we have to determine all involutions 
$\theta^r$ on $\g^r = \so_{n+1}(\R)$ commuting with 
$\tau^r$ and $\sigma^r$. 
As any such involution corresponds to an isometric 
involution on the sphere $\bS^n$, it is induced by an 
element $\gamma \in \OO_{n+1}(\R)$. 
That $\gamma$ commutes with $\tau^r$ means that 
$\gamma r_1 \gamma = \pm r_1$ (note that $\Ad(-\1) = \id_{\g^r}$). 
If $\gamma r_1 \gamma = - r_1$, then 
$\gamma$ exchanges the two eigenspaces of $r_1$. 
But, as $n > 1$,  the two eigenspaces of $r_1$ have different dimensions, 
so that we must have 
$\gamma r_1 \gamma = r_1$, i.e., $\gamma$ commutes with $r_1$. 
We likewise see that $\gamma$ commutes with $r_{n+1}$. Hence 
$\gamma$ is conjugate under the centralizer of $r_1$ and $r_{n+1}$ to an 
element of the form 
\[ \gamma = \diag(\gamma_1, \ldots, \gamma_{n+1}), \qquad \gamma_j = \pm 1.\] 
Since $\pm \gamma$ induce the same automorphism on $\g^r$, we 
may assume that $\gamma_1 = 1$. Permuting the entries of $\gamma$, we may 
now assume that there exists a $p \leq n+1$ such that 
we have for $q := n - p$  
\[ \gamma_1= \gamma_2 = \ldots = \gamma_{p} = 1, \qquad 
\gamma_{p+1}  = \ldots = \gamma_{p+q} = -1.\] 
Depending on $\gamma_{n+1}$, we find the two Lie algebras: 
\[
\g \cong  \begin{cases} 
\so_{p+1,q}(\R) & \text{ for } \gamma_{n+1} = 1 \\
\so_{p,q+1}(\R) & \text{ for } \gamma_{n+1}= -1,    
\end{cases}\qquad\qquad \mbox{ where } \qquad n  = p + q. \]
For $V := \R^{n+1}$, these Lie algebras can be realized naturally on the real 
linear subspace 
$V^c := V^\gamma \oplus i V^{-\gamma} \subeq \C^{n+1}$, on which they 
preserve a form of signature $(p+1, q)$, resp., $(p,q+1)$, according to 
$\dim V^{-\gamma} = q$ and $q+1$, respectively. 
This completes the proof of the classification theorem. 
\end{proof}

 \end{document}